\documentclass[11pt]{article}
\usepackage{amsmath,amscd,amsfonts,amssymb,amsthm,enumerate,
textcomp,xr,makeidx,microtype,a4wide,indentfirst}
\usepackage[mathscr]{eucal}

\usepackage[english]{babel}
\usepackage[all,2cell]{xy}
\usepackage[colorlinks,plainpages,urlcolor=blue,linkcolor=reference,citecolor=citation]{hyperref}
\usepackage[abbrev,lite,alphabetic]{amsrefs}

\usepackage{color}
\definecolor{reference}{rgb}{.20,.60,.22}
\definecolor{citation}{rgb}{0,.40,.80}

\UseTwocells
\xyoption{2cell}{\UseTwocells}
\hypersetup{linktocpage}

\makeatletter
\DeclareRobustCommand{\em}{%
  \@nomath\em \if b\expandafter\@car\f@series\@nil
  \normalfont \else \bfseries \fi}
\makeatother

\theoremstyle{definition}

\newtheorem{theorem}{Theorem}[section]
\newtheorem{lemma}[theorem]{Lemma}
\newtheorem{defn}[theorem]{Definition}

\newtheorem{prop}[theorem]{Proposition}
\newtheorem{zam}[theorem]{Remark}

\newtheorem{exs}[theorem]{Example}
\newtheorem{cor}[theorem]{Corollary}

\newtheorem*{Conv}{Conventions}

\newtheorem*{Ackn}{Acknowledgments}

\newcommand{{\Vect}}{\operatorname{\sf Vect}}

\newcommand{\Pic}{\operatorname{Pic}}
\newcommand{\NS}{\operatorname{NS}}
\newcommand{\CH}{\operatorname{CH}}
\renewcommand{\r}{\rightarrow}
\newcommand{\xr}{\xrightarrow}
\newcommand{\Spec}{\operatorname{Spec}}
\newcommand{\bZ}{\mathbb Z}
\newcommand{\Q}{\mathbb Q}
\newcommand{\bQ}{\Q}

\renewcommand{\a}{\alpha}
\renewcommand{\b}{\beta}
\renewcommand{\exp}{\operatorname{exp}}
\newcommand{\id}{\operatorname{Id}}

\DeclareMathOperator{\End}{End}

\begin{document}
\title{{\bfseries Chow Groups of Abelian Varieties and Beilinson's Conjecture}}
\date{}
\author{Bogdan~Zavyalov}
\maketitle
\begin{abstract}
\noindent In the present paper we introduce the property AA of a subsemigroup of the endomorphism semigroup of an abelian variety, which holds for semigroup of endomorphisms of an abelian variety defined over a number field, and show that the orbit of any cycle under a semigroup with property AA in the Chow group $\otimes\Q$ has finite dimensional span.
\end{abstract}
\tableofcontents

\section*{Introduction}

One of the most celebrated conjectures in 20th century algebraic geometry is a conjecture of Beilinson, which in particular predicts that for any variety $X$ defined over a number field the group $\CH(X)\otimes\Q$ has finite dimension. This conjectural picture is completely opposite to the one in complex setting, e.~g., Mumford proved that the Chow group of a smooth complex surface posessing a nonzero holomorphic 2-form is infinitely generated in any reasonable sense. The most na\"{\i}ve way to produce a counterexample to this conjecture is to build a variety $X$ over a number field together with a cycle on it such that its orbit w.~r.~t. {\it all} the endomorphisms of $X$ in the Chow group $\otimes\Q$ contains infinitely many linearly independent elements.

The goal of the present paper is to show that this attempt fails at least for $X$ an abelian variety. Although this fact is arithmetic in its nature, we show indeed that orbit of any cycle under any subsemigroup with the property AA (to be defined in the paper) of endomorphism semigroup (w.~r.~t. composition) has finite dimensional span in $\CH\otimes\Q$. This has some interesting applications even in the case of algebraically closed field. The proof of the fact that property AA holds for the entire endomorphism semigroup over number field is due to Mordell--Weil theorem.

The first part of the paper is a reminder of Fourier--Mukai transform developed by Beauville in his paper \cite{Beau}. The prove itself is given in the second part. The Fourier--Mukai transform allows us to reduce the case of an arbitrary cycle on $A$ to the case of the Poincar\'e line bundle on $A \times \widehat{A}$, and next we deal with the case of divisors explicitly.

\medskip

\begin{Ackn} We would like to thank Marat Rovinsky for suggesting the problem and numerous helpful discussions. Also, we would like to thank Rodion D\'{e}ev for reviewing the first draft of the paper and Dmitrii Pirozhkov without whose assistance this paper might have not been written at all. The research was supported in part by Dobrushin stipend.
\end{Ackn}

\medskip
\begin{Conv}
~
\begin{enumerate}
\item For any abelian variety $A$ and a $k$-rational point $a\in A(k)$ we will denote by $t_a: A \to A$ the translation of $A$ by $a$. We will denote the dual of the abelian variety $A$ by $\widehat{A}$ and the Poincar\'e line bundle by $L \in \Pic(A \times \widehat{A}).$ Moreover, we will denote the first Chern class of $L$ by $l$.
\item For any smooth projective variety $X$ over a field $k$ we will denote by $\CH(X):= \bigoplus_i \CH^i(X)$ the direct sum of all its Chow groups. The cup product enhances the abelian group $\CH(X)$ with a structure of a graded ring. For any proper map $f:X \r Y$ we will denote by $f_*$ the pushforward map on Chow rings (resp. groups) $f_*: \CH(X) \r \CH(Y)$ (resp. $f_*: \CH^i(X) \r \CH^i(Y)$). For any flat morphism $g: X \r Y$ we will denote by $g^*$ the pullback map on Chow rings (resp. groups) $g^*: \CH(Y) \r \CH(X)$ (resp. $g^*: \CH^i(Y) \r \CH^i(X)$). 
\item All the products of varieties $X \times Y$ will actually mean $X \times_{\Spec k} Y$. By $p_{i_1, \dots, i_n}: \prod_{j:=1}^m X_j \r \prod_{k:=1}^{n} X_{i_k}$ we will always denote the corresponding projection map.
\item For any variety $A$ over a field $k$ we will denote by $\End(A)$ the semigroup of its endomorphisms over $\Spec(k)$ w.~r.~t. composition. For an abelian variety $A$ we will denote by $\End_0(A)$ its ring of group endomorphisms over a base field. Note that $\End(A)$ has also a group structure w.~r.~t. addition on $A$, but we will not deal with it (unless the opposite is mentioned explicitly).
\item Subscript $_{\bQ}$ will always mean tensoring by $\bQ$. For example, for any abelian group (resp. ring) $R$, $R_{\bQ}$ will be a $\bQ$-vector space (resp. $\bQ$-algebra) $R\otimes_{\bZ} \bQ$.
\end{enumerate}
\end{Conv}

\section{Prelimanaries}

In this section we will provide the reader with all the well-known facts that will be important in the paper.

\subsection{Picard Group of Abelian Varieties}

 All the results in this section are well-known and could be found in any textbook on abelian varieties, possibly except for Proposition $\ref{split},$ which is not always formulated in such a manner.

Let $A$ be an abelian variety of dimension $g$ over a field $k$. We will need several facts about the structure of its Picard group.

Basic fact from algebraic geometry states that there is a short exact sequence
$$
 0 \r \Pic^0(A) \r \Pic(A) \r  \NS(A) \r 0,
$$
where $\Pic^0(A)$ is a group of algebraically trivial cycles, and $\NS(A)$ is a group of cycles modulo algebraic equivalence. Note that $\NS(A)$ is finitely generated $\bZ$-module for any smooth projective variety. Also, the semigroup of endomorphisms $G:=\End(A)$ acts on each term of this short exact sequence making it into a short exact sequence of $G$-modules. The crucial fact is that after tensoring by $\bQ$ it splits as a short exact sequence of $G_0:=\End_0(A)$-modules.

\begin{defn} A divisor $D \in \Pic(A)$ is called {\it symmetric}, if $[-1]^*D=D.$, and {\it antisymmetric}, if $[-1]^*D=-D.$ Let us denote the group of symmetric (resp. antisymmetric) divisors by $\Pic^+(A)$ (resp. $\Pic^{-}(A)$).
\end{defn}

The following lemmata give us a very good characterization of symmetric and antisymmetric divisors.

\begin{lemma}{\label{sym}} Let $A$ be an abelian variety over a field $k$, $D \in \Pic(A)$. Then the following are equivalent:
\begin{enumerate}
\item $D$ is symmetric,
\item $[n]^*D=n^2D$ for all $[n] \in \bZ$.
\end{enumerate}

\end{lemma}
\begin{proof}
\cite[Cor. 5.4]{Milne}
\end{proof}

\begin{lemma}{\label{antisym}} Let $A$ be an abelian variety over a field $k$, $D \in \Pic(A)$. Then the following are equivalent:
\begin{enumerate}
\item $D$ is antisymmetric,
\item $[n]^*D=nD$ for all $[n] \in \bZ$,
\item For every $f,g \in \End_0(A)$ we have $(f+g)^*(D)=f^*(D) + g^*(D)$,
\item $D \in \Pic^0(A)$.
\end{enumerate}
Additionally, if $D$ is antisymmetric, then for each $a\in A(k)$ one has $\ t_a^*(D)=D$. 

\end{lemma}
\begin{proof}
\cite[Cor. 5.4. + Remark 8.5]{Milne}
\end{proof}

These lemmata allows us to construct a splitting of the short exact sequence
\begin{multline}
\label{very_good_sequence}\\ 0 \r \Pic^0(A)_{\bQ} \r \Pic(A)_{\bQ} \xr{r} NS(A)_{\bQ} \r 0 \\  
\end{multline}
\begin{prop}{\label{split}} There is a canonical splitting of \ref{very_good_sequence} as $G_0$-modules given by a map $ \phi: \Pic(A)_{\bQ} \r \Pic^0(A)_{\bQ}$ defined as $\phi(D)=(D-[-1]^*D)/2$. Moreover, $\Pic(A)_{\bQ}\simeq \Pic^+(A)_{\bQ} \oplus \Pic^-(A)_{\bQ}$ as $G_0$-modules as well as $\Pic^+(A)_{\bQ} \simeq \Pic^0(A)_{\bQ}, \ \Pic^-(A)_{\bQ} \simeq \Pic^0(A)_{\bQ}$
\end{prop}
\begin{proof}
First of all, note that the lemma \ref{antisym} says that a natural inclusion of $\Pic^-(A)_{\bQ}$ into $\Pic^0(A)_{\bQ}$ gives us an isomorphism of them as $G_0$-modules. Secondly, we have two projections $p_{\pm}:\Pic(A)_{\bQ} \to \Pic^{\pm}(A)_{\bQ}$, namely $p_{\pm}(D)=\frac{D\pm [-1]^*D}{2}$. It is easy to check that these morphisms are morphisms of $G_0$-modules, so it provides us with a decomposition $\Pic(A)_{\bQ}\simeq \Pic^+(A)_{\bQ} \oplus \Pic^-(A)_{\bQ}$. This implies that $\phi$ is a $G_0$-module section of the short exact sequence. Finally, $\ker r = \Pic^0(A)_{\bQ} = \Pic^-(A)_{\bQ}$ and $\Pic^+(A)_{\bQ} \cap \Pic^-(A)_{\bQ} = 0.$ Thus, by exactness of \ref{very_good_sequence} and the fact that all morphisms are morphisms of $G_0$-modules, we conclude that $\Pic^+(A)_{\bQ} \simeq \NS(A)_{\bQ}.$
\end{proof}

\begin{zam} In particular this statement claims that $[n]^*(\a)=n^2\a$ for every $\a \in \NS(A)_{\bQ}.$ Also, the proof actually shows that there is a morphism of $G_0$-modules $\psi: \NS(A)_{\bQ} \to \Pic(A)_{\bQ}$ defined by $\psi(\a)=(D+[-1]^*D)/2,$ where $D \in \Pic(A)_{\bQ}$ is any representative of $\a \in \NS(A)_{\bQ}.$
\end{zam}

\subsection{Fourier--Mukai Tranform}

\begin{defn}{\label{FM}} Let $A$ be an abelian variety. The Fourier--Mukai transform is a map $F: \CH(A)_{\bQ} \to \CH(\widehat{A})_{\bQ}$ defined by $F(\alpha)= p_{2*}(p_1^*(\alpha) \cdot \exp(l))$ for every $\a \in \CH(A)_{\bQ}$. By $\exp$ we undestand the map defined by formal power series $\exp(t)=1+t+\frac{t^2}{2}+ \dots$. Note that $\exp(l)$ is well-defined because $A$ is of finite dimension.

By $\widehat{F}$ we will denote the Fourier--Mukai tranform on the dual of abelian variety $\widehat{F} \colon \CH(\widehat{A})_{\bQ} \to \CH(A)_{\bQ}.$
\end{defn}

The following theorem is essential in the proof of the Main Theorem. Namely, it will allow us to reduce the general case to the case of divisors.

\begin{theorem}{\label{inv}} Let $A$ be an abelian variety of dimension $g$. Then the following formula holds
$$
\widehat{F}\circ F = (-1)^g[-1]^*
$$
\end{theorem}
\begin{proof}
\cite[F2 p.647]{Beau}
\end{proof}

%\begin{prop}{\label{B1}} Let $A$ be an abelian variety of dimension $g$. Fix $\a \in \CH^p(A)_{\bQ}$ such that $F(\a)=\sum \eta_q%$ with $\eta_q \in \CH^q(\widehat{A})_{\bQ}.$ Then for any $n \in \bZ$ we have $[n]^*(\eta_q)=n^{g-p+q}.$
%\end{prop}
%\begin{proof}
%Beauville
%\end{proof}

%\begin{prop}{\label{B2}} Let $A$ be an abelian variety of dimension $g$ and $\a \in \CH(\widehat{A})_{\bQ}$. Assume that for every $n \in \bZ$ the following formula holds $[n]^*\a=n^{2q-s}\a,$ then $F(\a) \in \CH^{g-q+s}(\widehat{A})_{\bQ}$
%\end{prop}
%\begin{proof}
%Beauville
%\end{proof}

\section{The Main Theorem}

Before proving the Main Theorem we have to define some class of subsemigroups of the endomorphism semigroup of abelian variety for which it holds. The definition of such subsemigroups will be a little bit technical, but we will provide the reader with some interesting examples of such subsemigroups. 

Recall that the semigroup of endomorphisms $G=\End(A)$ of any abelian variety is a semidirect product of the semigroup of its group endomorphisms $G_0=\End_0(A)$ and translations. If we denote the latter one by $T$, then $G=T \rtimes G_0$.

%We have an inclusion $H/(H \cap G_0) \hookrightarrow G/G_0=T,$ therefore, $H/(H \cap G_0)$ inherits a natural structure of abelian semigroup. 

\begin{defn} We say that a subsemigroup $H \subset G$ has {\it property AA}, if there is a finite number of $k$-rational points $a_1, \dots , a_n$ such that any element of $h \in H$ can be written as $h=f \circ t_{a_n}^{l_n} \circ \dots \circ t_{a_1}^{l_1},$ where $f \in G_0$ and $l_i \in \mathbb Z$.
\end{defn}

\begin{exs} 
\begin{enumerate}
\item The subsemigroup $G_0$ has property AA for trivial reasons.
\item Every finitely generated subgroup $T_0 \subset T$ has property AA.
\item Fix some finitely generated subgroup $T_0 \subset T$. Then a subsemigroup $T_0 \rtimes G_0$ has property AA.
\item (the most interesting example)\label{vajno} Let $A$ be an abelian variety over a number field $k$. Then the semigroup of all endomorphisms of $A$ satisfies the property AA. The Mordell--Weil theorem (\cite[Chapter 4]{Serre}) claims that for any abelian variety over number field $A(k)$ is a finitely generated group. It reduces this example to the third one.
\end{enumerate}
\end{exs}

\begin{theorem}{\label{mtheorem}}[The Main Theorem] Let $A$ be an abelian variety over a field $k$. Fix some subsemigroup $H \subset G$ that satisfies the property AA. Then for any cycle $\a \in \CH^p(A)_{\bQ}$ its orbit under the action of the semigroup $H$ spans a finite-dimensional $\bQ$-vector space.
\end{theorem}
\begin{proof} 

We are going to prove the theorem in three steps. Firstly, we will prove it for $H=G_0$ and $p=1$. Secondly, we will reduce the case of any subsemigroup with property AA and $p=1$ to the case $H=G_0.$ Finally, we will reduce the general case to the case $p=1$. 

\textbf{Step 1.} Assume that $p=1$ and $H=G_0.$ We will prove that for every finite-dimensional vector space $V \subset \Pic(A)_{\bQ}$ its orbit under the action of $G_0$ spans finite-dimensional vector space.

According to Proposition \ref{split}, we have a split short exact sequence of $G_0$-modules
$$
 0 \to \Pic^0(A)_{\bQ} \to \Pic(A)_{\bQ} \to \NS(A)_{\bQ} \to 0.
$$ 

Since this short exact sequence of $G_0$-modules splits and $\NS(A)_{\bQ}$ is of finite dimension, we can assume that $V \subset \Pic^0(A)_{\bQ}$. From now on we will consider $G_0$ as a group under addition instead of considering it as semigroup under composition. Lemma \ref{antisym} tells us that the action of $G_0$ is linear (w.r.t. to this group structure) on $\Pic^0(A)_{\bQ}$. Combined with the fact that $G_0$ is a finitely generated group (\cite[Prop. 10.5 + Lemma 10.6]{Milne}), we conclude that the orbit of the vector subspace $V$ under the action of $G_0$ spans a finite-dimensional vector space.

\textbf{Step 2.} Now assume that $p=1$, but semigroup $H$ is any subsemigroup of $G$ with the property AA. According to the definition, we can choose a finite number of $k$-rational points $a_1, \dots, a_n \in A(k)$ such that any element of $g \in H$ can be written in the following form $g=f \circ t_{a_n}^{l_n} \circ \dots  \circ t_{a_1}^{l_1}$ with $f \in G_0$. We know \cite[Lemma 8.8]{Milne} that $\b_i:=t_{a_i}^*(\a)-\a \in \Pic^0(A)_{\bQ}$ for any $0<i \leq n.$ Since all $t_{a_i}$ commute and $t_a^*(\b)-\b=0$ for any $\b\in \Pic^0(A)_{\bQ}$(Lemma \ref{antisym}), we conclude that $(t_{a_n}^{l_n})^* \circ \dots  \circ (t_{a_1}^{l_1})^* (\a)= \a -l_1\b_1 - \dots -l_n\b_n.$ In particular, the vector space $V$ generated by the orbit of $\alpha$ under the action of endomorphisms of the form $t_{a_n}^{l_n} \circ \dots  \circ t_{a_1}^{l_1}$ is of finite dimension. Therefore, since $H$ has property AA, it suffices to show that the orbit of $V$ under the action of $G_0$ spans finite-dimensional vector space. But it has been already done in the Step 1.

\textbf{Step 3.} Finally, we are going to reduce the general case to the case of a divisor on the abelian variety $A\times \widehat{A}.$

Theorem \ref{inv} says that $\a = (-1)^g[-1]^*\widehat{F}(F(\a)).$ Let $F(\a)=\sum_{q=1}^{g} \eta_q$ with $\eta_q \in \CH^q(\widehat{A})_{\bQ}.$ Again, choose finite number of $k$-rational points $a_1, \dots, a_n \in A(k)$ such that any element of $g \in H$ can be written in the following form $g=f \circ t_{a_n}^{l_n} \circ \dots  \circ t_{a_1}^{l_1}$ with $f \in G_0$. Choose such a representation for each $g \in H$ and denote by $g':=f \circ t_{-a_n}^{l_n} \circ \dots  \circ t_{-a_1}^{l_1}.$ Also we will denote by $\overline g$ the morphism $g \times \id_{\widehat{A}}: A\times \widehat{A} \to A\times \widehat{A}$. Now, for any $g \in H$ we have
$$
g^*(\a)=(-1)^g(g^* \circ [-1]^* \circ \widehat{F} \circ F )(\a)=(-1)^g(g^* \circ [-1]^* \circ \widehat{F})(\sum_{q=1}^g \eta_q)= (-1)^g\sum_{q=1}^{g}g^*([-1]^* \circ \widehat{F}(\eta_q))
$$

This expression allows us to conclude that in order to prove the Theoreom for a cycle $\a$ it is sufficient to prove it for every $[-1]^* \circ \widehat{F}(\eta_q)$. From now on we are going to prove the Theorem in this case.

Note that for any $g \in H$ we have that $g^*\circ [-1]^*=[-1]^*g'$. Also the base change formula applied to the following Cartesian diagram asserts that $f^* \circ p_{1*}=p_{1*}\circ(\overline f)$ for every $f \in G$.
$$
\begin{CD}
 A \times \widehat{A} @>\overline f>> A \times \widehat{A}\\
@Vp_1VV @Vp_1VV \\
A @>f>> A 
\end{CD}
$$

We conclude that

\begin{equation*}
\begin{split}
g^*([-1]^* \circ \widehat{F}(\eta_q))=g^* \circ [-1]^*(p_{1*}(p_2^*(\eta_q) \cdot \exp(l)))=[-1]^*\circ g'^*(p_{1*}(p_2^*(\eta_q) \cdot \exp(l)))= \\
=[-1]^*\circ p_{1*}(\overline g'^*(p_2^*(\eta_q) \cdot \exp(l)))=[-1]^* \circ p_{1*}(p_2^*(\eta_q) \cdot \exp(\overline g'^*(l))) =\\
=[-1]^* \circ p_{1*}(p_2^*(\eta_q) \cdot (\sum_{i=1}^{\infty}\overline g'^*(l^n)))=\sum_{i=1}^{\infty} ([-1]^* \circ p_{1*}(p_2^*(\eta_q) \cdot \overline g'^*(l^n)))
\end{split}
\end{equation*}

The sum is actually finite because $A\times \widehat{A}$ is a variety of dimension $2g$, so $l^n=0$ for $n>2g$. Since cup-product, pullbacks and pushforwards are linear maps, it is enough to show that the orbit of each $l^n$ under the action of elements of the form $\overline g'^*$ spans a finite-dimensional vector space. Moreover, since $\overline g'^*(l^n)=\overline g'^*(l)^n$ it is enough to show it only for $l$. Every element of the form $\overline g'$, by definition, can be written in the form 
$$
\overline g'=(f \times \id) \circ (t^{k_n}_{-a_n} \times \id) \circ \dots \circ (t^{k_1}_{-a_1} \times \id)=(f \times \id) \circ t^{k_n}_{(-a_n,0)}  \circ \dots \circ t^{k_1}_{(-a_1,0)}
$$

Since $f\in \End_0(A)$, we have $f\times \id \in \End_0(A\times \widehat{A}).$ Therefore, the subset $$H'=\{f \in \End(A\times \widehat{A})| \exists g \in H \colon f = \overline g'\}$$ is a subsemigroup and has property AA. Therefore, we can apply the Step 2 to $H'$ and $l$ and finish the proof.
\end{proof}

\begin{cor} Let $A$ be an abelian variety over any field $k$. Fix some cycle $\b \in \CH^p(A)_{\bQ}$ and a $k$-rational point $x \in A(k),$ then the set $\{\a, t_x^*(\a), (t^*_x)^2(\a), \dots\}$ spans finite dimensional vector space.
\end{cor}
\begin{proof}
Apply Theorem \ref{mtheorem} to $H=\bZ t_x^*$ and $\a=\b$.
\end{proof} 

\begin{cor} Let $A$ be an abelian variety over any field $k$. Fix some cycle $\b \in \CH^p(A)_{\bQ}$, then the orbit of $\b$ under the action of the semigroup $\End_0(A)$ spans finite dimensional vector space.
\end{cor}
\begin{proof}
Apply Theorem \ref{mtheorem} to $H=\End_0(A)$ and $\a=\b$.
\end{proof}

\begin{cor} Let $A$ be an abelian variety over number $k$. Fix some cycle $\b \in \CH^p(A)_{\bQ}$. Then the orbit of $\b$ under the action of the group of all the endomorphism semigoup $\End(A)$ spans a finite-dimensional vector space.
\end{cor}
\begin{proof}
According to Example \ref{vajno}, $\End(A)$ has property AA. Thus we can apply Theorem \ref{mtheorem} to $H=\End(A)$ and $\a=\b$.
\end{proof}

\end{document}